\documentclass{amsart}

\usepackage{amsmath,amssymb,amsfonts}
\usepackage{color,xcolor}
\usepackage{graphicx}
\usepackage{mathtools}
\usepackage{tikz}
\usepackage{hyperref}

\newcommand{\cU}{\mathcal{U}}

\newcommand{\ang}[1]{\langle #1\rangle}

\newcommand{\Id}{\mathcal{I}}

\newtheorem{theorem}{Theorem}[section]
\newtheorem{lemma}[theorem]{Lemma}

\theoremstyle{definition}

\newtheorem{question}[theorem]{Question}

\title{Tukey order among $F_{\sigma}$ ideals}

\begin{document}

\author[J.\,He]{Jialiang He }
\address{College of Mathematics, Sichuan University, Chengdu, Sichuan, 610064, China}
\curraddr{}
\email{jialianghe@scu.edu.cn}
\thanks{The first author was partially supported by NSF of China
grants 11801386 and 11771311.}

\author[M.\,Hru\v{s}\'ak]{Michael Hru\v{s}\'ak}
\address{Centro de Ciencias Matem\'aticas, UNAM, Morelia, Mexico}
\email{michael@matmor.unam.mx}
\thanks{The second author was partially supported by PAPIIT grants IN 100317 and IN 104220, and a CONACyT grant A1-S-16164.}

\author[D.\,Rojas-Rebolledo]{Diego Rojas-Rebolledo}
\address{Department of Mathematics and Computing Sciences, Saint Mary's University, Halifax, Canada}
\email{Diego.Rojas@smu.ca}

\author[S.\,Solecki]{ S{\l}awomir Solecki}
\address{Department  of Mathematics, Cornell University, Ithaca, NY 14853, USA}
\email{ssolecki@cornell.edu}
\thanks{The fourth author was partially supported by NSF grant DMS-1954069.}


\begin{abstract} We investigate the Tukey order in the class of $F_\sigma$ ideals of subsets of $\omega$.
We show that no nontrivial $F_\sigma$ ideal is Tukey below a $G_\delta$ ideal of compact sets.
We introduce the notions of flat ideals and gradually flat ideals. We prove a dichotomy theorem for flat ideals isolating gradual flatness as the side
of the dichotomy that is structurally good. We give diverse characterizations of gradual flatness among flat ideals using Tukey reductions and games.
For example, we show that
gradually flat ideals are precisely those flat ideals that are Tukey below the ideal of density zero sets.
\end{abstract}

\maketitle

\section{Introduction}

The present paper is a contribution to the study of the structure of the {\bf Tukey order} on definable directed sets; for background, see
\cite{fremlin-91, Fr2, hrusak-filters, isbell, LV, Ma, schmidt, solecki, ST04, todorcevic, tukey}.
For two directed orders $P$ and $Q$, we write
$P\leq_T  Q$ if there is a function $h: P \to  Q$ sending unbounded sets to unbounded sets.\footnote{In this paper, {\em bounded} always means bounded from above, and {\em unbounded} means unbounded from above.}
Such an $h$ is called a {\bf Tukey map} and $P$ is said to be {\bf Tukey reducible to} $Q$. We write
$P\equiv_T Q$ if both $P\leq_T Q$ and $Q\leq_T P$.
Tukey order was introduced in \cite{tukey} to study convergence of nets and was recast by Schmidt \cite{schmidt} and Isbell \cite{isbell} to compare cofinal types of directed orders.
The main class of directed orders considered by us are
$F_\sigma$ ideals of subsets of $\omega$ taken with inclusion as the order relation.

Recall that a subset of a metric space is $F_\sigma$ if it is the union of a countable family of closed sets; it is $G_\delta$ if it is 
the intersection of a countable family of open sets. To be $F_\sigma$ for an ideal of subsets of $\omega$ means $F_\sigma$ 
with respect to 
the product topology on ${\mathcal P}(\omega)= \{ 0,1\}^\omega$, with $\{ 0,1\}$ given the discrete topology. 
To be $G_\delta$ for an ideal of compact subsets of 
a Polish space $X$ means $G_\delta$ with respect to the Vietoris topology on the space of all compact subsets of $X$.

An important role in the study of ideals of subsets of $\omega$ is played by lower semicontinuous submeasures.
A function $\varphi\colon {\mathcal P}(\omega)\to [0,\infty]$ is a {\bf submeasure} if
\begin{enumerate}
\item[---] $\varphi(\emptyset)=0$,

\item[---] $\varphi(a)\leq\varphi(b)$, for $a\subseteq b$,

\item[---] $\varphi(a\cup b)\leq \varphi(a) +\varphi(b)$, for all $a,b\subseteq\omega$.
\end{enumerate}
A submeasure $\varphi$ is {\bf lower semicontinuous} ({\bf lsc}) if
\begin{enumerate}
\item[---] $\varphi(a)=\lim_{n\to\infty} \varphi(a\cap [0,n])$, for all $a\subseteq \omega$.
\end{enumerate}
By a theorem of Mazur \cite{Maz}, an ideal $\Id$ of subsets of $\omega$ is $F_\sigma$ precisely when there is a lower semicontinuous submeasure
$\varphi$ such that
$$
\mathcal I={\rm Fin} (\varphi)=\{a\subseteq\omega: \varphi(a)<\infty\}.
$$

We study the place  $F_\sigma$ ideals of subsets of $\omega$ occupy in the Tukey order.
As a point of reference we use the following, now standard, diagram which summarizes
the known Tukey reductions among
the well studied directed sets; see \cite{fremlin-91, Fr2, isbell, LV, matrai-ken, moore-solecki, solecki-11, solecki, ST04}.
As usual, in the diagram, an arrow denotes reduction and the absence of arrows non-reduction.
The partial orders in the diagram
are defined below it.
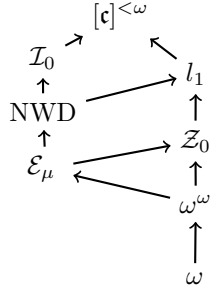
\begin{center}
 \begin{tikzpicture}
\node (E1) at (0:1cm) {$\omega^\omega$};
\node (E2) at (130: 1.6cm) {$\mathrm{NWD}$};
\node (E9) at (118: 2.2cm ) {$\mathcal I_0$};
\node (E3) at (60:2cm) {$l_1$};
\node (E4) at (90: 2.5cm) {$[\mathfrak c]^{<\omega}$};
\node (E6) at (-45:1.4cm) {$\omega$};

\node (E7) at  (40: 1.3cm) {$\mathcal Z_0$};
\node (E8) at (155: 1.1cm) {$\mathcal E_\mu~$};
\draw[->,thick]
(E8) edge (E7)
(E2) edge (E3)
(E3) edge (E4)
(E6) edge (E1)

(E1) edge (E7)
(E7) edge (E3)
(E8) edge (E2)
(E1) edge (E8)
(E2) edge (E9)
(E9) edge (E4)
;
\end{tikzpicture}

\end{center}
The partial orders in the diagram are Tukey equivalent to ideals of subsets of $\omega$ or to ideals of compact subsets of compact metric spaces.
The order $[\mathfrak c]^{<\omega}$ of finite subsets of $2^\omega$ taken with inclusion as the order relation is Tukey equivalent to an $F_\sigma$ ideal
of subsets of $\omega$, see \cite[Section 5, Proposition 3]{LV}. 
The \emph{summable ideal}  and the \emph{density zero ideal} are defined, respectively, as
\[
l_1=\left\{A\subseteq \omega: \sum_{n\in A}\frac1{n+1}<\infty\right\}\;\hbox{ and }\;\mathcal{Z}_0=\left\{A\subseteq\omega: \lim_n \frac{|A\cap n|}{n}=0\right\}.
\]
The ideal $\mathrm{Fin}$ of finite subsets of $\omega$ is Tukey equivalent to $\omega$, and
$$
\omega^\omega\equiv_T\emptyset\times {\rm{fin}}=\left\{A\subseteq\omega\times\omega: \ (\forall n\in\omega) |\{m\in\omega: (n,m)\in A\}|<\omega \right\},
$$
The ideal $\mathrm{NWD}$ is
the ideal of closed nowhere dense subsets of $2^\omega$, and $\mathcal E_\mu$ is the ideal of closed measure zero subsets of $2^\omega$. To define $\mathcal I_0$, denote by
$\mathcal S$  the collection of all block sequences of finite partial functions from $\omega$ into $2$  of even or infinite length, and  for $\overline s\in \mathcal S$ let
$$[\overline s]=\{ x\in 2^\omega: \ \forall i\  (s_{2i}\subseteq x \text{ or } s_{2i+1}\subseteq x)\}$$ and
$$
\mathcal I_0=\{K\subseteq 2^\omega: K\text{ is compact and } K\cap [\overline s] \text{ is nowhere dense in } [\overline s] \text{ for all }\overline s\in\mathcal S\}.
$$
The only $F_\sigma$ ideals appearing in the diagram are $\omega$, $[\mathfrak c]^{<\omega}$ and $l_1$.

The left column of the above diagram, or more accurately the class of all $G_\delta$ ideals of compact subsets of Polish spaces, is often referred to as
the {\bf category leaf}. The right column, or more precisely the class of all analytic P-ideals of subsets of $\omega$, is called the {\bf measure leaf}.

In the present paper, we study
a third leaf, one could call the {\bf $F_\sigma$ leaf}, consisting of $F_\sigma$ ideals of subsets of $\omega$.
First, we look into the relationship between $F_\sigma$ ideals
and the category leaf. We show in Theorem~\ref{thm9} that, except for trivial situations, $F_\sigma$ ideals are not reducible to $G_\delta$ ideals of compact subsets of Polish spaces.
Next, we restrict our
attention to a class of $F_\sigma$ ideals that we call flat. We prove that flat ideals are not basic, except, again, for trivial situations. This and the previous result indicate
distinctness of the leaf studied in this paper from the other two leaves.
Within the class of flat ideals, we show a dichotomy, Theorem~\ref{dichotomy}, saying that a flat ideal is Tukey equivalent to the top order in
the diagram, namely $[\mathfrak c]^{<\omega}$, or it has a structural property, which we call gradual flatness. Then, in Theorem~\ref{thm8}, we compare gradual flatness with
the measure leaf by
showing that such ideals are reducible to the density zero ideal ${\mathcal Z}_0$. In fact, Theorem~\ref{thm8} shows that gradual flatness is a robust property, as it turns out to be
equivalent to a number of diverse conditions. In the remarks following Theorem~\ref{thm8}, 
we point out the high complexity of the structure of the Tukey reduction among gradually flat ideals.

\section{$F_{\sigma}$ ideals of subsets of $\omega$ and $G_\delta$ ideals of compact sets}

We prove a theorem that is an indication of ``orthogonality" between $F_\sigma$ ideals of subsets of $\omega$ and $G_\delta$ 
ideals of compact subsets of Polish spaces. By
\cite[Corollary~6.4]{ST04}, a similar
relation holds between analytic P-ideals of subsets of $\omega$ and $G_\delta$ ideals of compact subsets of Polish spaces.

\begin{theorem}\label{thm9} If $\mathcal J$ is a $G_\delta$ ideal of compact subsets of a Polish space and $\mathcal I$ is an uncountably generated $F_\sigma$ ideal on $\omega$, then $\mathcal I\not\leq_T \mathcal J$.
\end{theorem}

For the proof of this theorem, we use the second of Fremlin's games introduced in~\cite{fremlin}.
Let $P\not=\emptyset$ be a partially ordered set. The Game $\Gamma_2(P)$ is the two-player game where,  setting $U_{-1}=P$, 
the $n$-th move of the game is described for both players as follows. Given  $U_{n-1}\subseteq P$,
Player I plays a countable cover $\mathcal{U}_n$ of $U_{n-1}$, 
and Player II responds with $U_n\in\mathcal{U}_n$ and a finite set $I_n\subseteq U_n$. 
(By  {\em countable cover} here we mean that $\mathcal{U}_n$ is a countable family of sets 
and $\bigcup\mathcal{U}_n =U_{n-1}$.)
The following diagram gives a pictorial representation of the $n$-th move. \\

\medskip

\begin{tabular}{cccccc}
\multicolumn{1}{c|}{$I$}  &$\mathcal U_{n}$ ctble cover of $U_{n-1}$   \\ \hline
\multicolumn{1}{c|}{$II$} & &$U_n\in \mathcal U_n$,  $I_n\in [U_n]^{<\omega}$  \\
\end{tabular}\\

\medskip

Player I wins the game if the set $\bigcup_{n\in\omega}I_n$ is bounded in $P$,  otherwise Player II wins.

The theorem follows directly from the following three lemmas. The first one is easy to prove and is contained in \cite{fremlin}.

\begin{lemma}[\cite{fremlin}]\label{lemma2} Let $P$ and $Q$ be partially ordered sets such that $P\leq_T Q$.
\begin{enumerate}
\item If Player I has  a winning strategy in $\Gamma_2(Q)$, then Player I also has a  winning strategy in $\Gamma_2(P)$.
\item If Player II has  a winning strategy in $\Gamma_2(P)$, then Player II also has a  winning strategy in $\Gamma_2(Q)$.
\end{enumerate}
  \end{lemma}

\begin{lemma} Player II has a winning strategy in $\Gamma_2(\Id)$ for every uncountably generated $F_\sigma$ ideal $\mathcal I$ on $\omega$.
\end{lemma}

\begin{proof} Let $\varphi$ be a lsc submeasure such that $\mathcal I={\rm Fin}(\varphi)$. A winning strategy for Player II can be described as follows. 
Start with $U_0=\mathcal I$. Once $\mathcal{U}_{n-1}$ has been played, as $\Id$ is not countably generated, choose $U_n\in\mathcal{U}_{n-1}$ which is not countably generated, and $I_n\in[U_n]^{<\omega}$ such that $\varphi(\bigcup I_n)>n$.
\end{proof}

\begin{lemma} Player I has a winning strategy in $\Gamma_2(\mathcal J)$ for every  $G_\delta$ ideal $\mathcal J$ of compact sets of  some Polish space $X$.
\end{lemma}

\begin{proof} Let $\delta$ be the Hausdorff  metric on $\mathcal K(X)$.  Since  $\mathcal J$ is $G_\delta$, let $\{ F_n: n\in \omega\}$ be a sequences of  closed subset of $\mathcal K(X)$  such that $\mathcal K(X)\setminus \mathcal J=\bigcup_{n\in\omega} F_n$.

We can describe a winning  strategy  for Player I as follows. At step $0$,  for each  $A\in \mathcal J$, choose
$\epsilon_A>0$ such that $B_{ \delta }( A,\epsilon_A)\cap F_0=\emptyset$.
Since $(\mathcal K(X),\delta)$ is Lindel\"{o}f, there exists a countable subcover
 $\{ B_\delta(A^1_i, \epsilon_{A^1_i}): i\in \omega\}$ of  $\mathcal J$.
Player I plays $\mathcal U_0=\{ B_\delta(A^1_i, \epsilon_{A^1_i}): i\in \omega\}$.
At step $n+1$, suppose  Player II plays $B_{\delta}(A^n_{i_n},\epsilon_{A^n_{i_n}})$ and  a finite subset $I_n\subseteq  B_{\delta}(A^n_{i_n},\epsilon_{A^n_{i_n}})\cap \mathcal J $. As in  step $1$,  for each  $A\in  B_{\delta}(A^n_{i_n},\epsilon_{A^n_{i_n}})\cap \mathcal J$, there is an $\epsilon_A\in (0, \frac {1}{n+1})$ such that $B_{ \delta }( A,\epsilon_A)\cap F_n=\emptyset$.  Player I plays a countable subcover  $\mathcal U_{n+1}=\{ B_\delta(A^{n+1}_i, \epsilon_{A^{n+1}_i}): i\in \omega\}$ of  $B_{\delta}(A^n_{i_n},\epsilon_{A^n_{i_n}})\cap \mathcal J$.

To see that the strategy is winning, note that the sequence $\{ A^n_{i_n}: n\in\omega\}$ is a Cauchy sequence in $\mathcal K(X)$ hence converges to some $A\in \mathcal K(X)$. By the construction, $A\not \in \bigcup_{n\in\omega} F_n$ so, $A\in\mathcal J$. As each $I_n$ is a finite subset of  $B_{\delta}(A^n_{i_n},\epsilon_{A^n_{i_n}})$, the sequence $\bigcup_{n\in\omega} I_n$ also converges to $A$, and as, 
by \cite[Theorem~3]{Kec}, $\mathcal J$ is a $\sigma$-ideal of compact sets, 
we get $A\cup \bigcup_{n\in\omega} \bigcup I_n \in \mathcal J$. (Theorem 3 in \cite{Kec} is stated only for $G_\delta$ ideals of compact subsets of compact metric spaces, but its proof works for $G_\delta$ ideals of compact subsets of arbitrary Polish spaces.)
\end{proof}

For more information on ideals of compact sets, the reader may consult \cite{KLW,MZ,matrai-ken, moore-solecki, ST04}.

\section{Flat ideals}\label{FGF1}

In this and subsequent sections, we restrict our attention to a broad subclass of $F_\sigma$ ideals that is defined in terms of submeasures.
A lsc submeasure $\varphi$ is {\bf flat} if, for each $M>0$, there exists $N>0$ such that, for each finite set
$a\in {\mathcal P}(\omega)$ with $\varphi(a)<N$, we have
\[
\exists j \ \forall b\in {\mathcal P}(\omega\setminus j) \
\Big(\varphi(b)< M\Rightarrow \varphi(a\cup b)< N\Big).
\]
We will sometimes say that $N$ witnesses flatness of $\varphi$ for $M$. 
An ideal is called {\bf flat} if it is of the form ${\rm Fin}(\varphi)$ for a flat submeasure $\varphi$.
We point out that fragmented ideals\footnote{ An $F_{\sigma}$ ideal $\Id={\rm Fin}(\varphi)$ is {\em fragmented} if there is a partition $(a_j)$ of $\omega$ into finite sets such that $\lim_j \varphi(a_j)=\infty$ and $\Id=\{b\subseteq \omega:\exists k\, \forall j\; \varphi(a_j\cap b)<k\}$.}, as introduced in \cite{HRZ}, are clearly flat.
Note that if $\Id={\rm Fin}(\varphi)= {\rm Fin}(\psi)$ and  $\varphi$  is flat, it  does not follow  that $\psi$ is flat. For example,
\[
[\omega]^{<\omega}={\rm Fin}(\varphi)={\rm Fin}(\psi), \hbox{ where }\varphi(A)=\sup A \hbox{ and }\psi(A)=|A|.
\]
Both $\varphi$ and $\psi$ are lsc submesures, $\varphi$ is flat, while $\psi$ is not.

It was proved in \cite{ST04} that both the measure leaf and the category leaf, that is, all analytic P-ideals of subsets of $\omega$ and all $G_\delta$ ideals of compact subsets of Polish spaces,
are included in a general class of partial
orders, called {\bf basic} partial orders. (This notion is defined below.) As shown in \cite{ST04}, a number of arguments related to Tukey reductions can be run for general basic orders. Here, however, we prove that,
unless a flat ideal is countably generated (so very simple) it is not basic. This result, along with Theorem~\ref{thm9}, highlights ``orthogonality" of the partial orders
considered in this paper with the previously studied classes of analytic P-ideals and $G_\delta$ ideals of compact sets.

In order to state our result, we recall the definition of basic orders from \cite[Section~3]{ST04}.
This definition involves a topology on a partial order; in the two cases mentioned above, of analytic P-ideals of subsets of $\omega$ 
and $G_\delta$ ideals of compact subsets of Polish spaces, the topologies
making the orders into basic orders are
the submeasure topology and the Vietoris topology, respectively; see \cite[Section~3]{ST04} for details.
A partial order $(P, \leq)$ with a metric separable topology on $P$ is {\bf basic} provided that
\begin{enumerate}
\item[---] any pair of elements of $P$ has the least upper bound with respect to $\leq$ and the binary operation of taking the least upper bound is continuous;

\item[---] each bounded sequence has a convergent subsequence;

\item[---] each convergent sequence has a bounded subsequence.
\end{enumerate}

A result analogous to Theorem~\ref{P:nonb} below was proved by Matr{\'a}i in \cite[Proposition~5.28]{Ma} for a specific family of ideals. Our theorem generalizes M{\'a}trai's
result; our proof expands on his approach. On the other hand, Theorem~\ref{P:nonb} strengthens, within the class of flat ideals, the general result
\cite[Corollary~4.2]{ST04}.

\begin{theorem}\label{P:nonb}
Let $\Id$ be a flat ideal. Then either $\Id$ is countably generated or $\Id$ is not basic with any topology on $\Id$.
\end{theorem}

\begin{proof} Let $\varphi$ be a flat submeasure with $\Id = {\rm Fin}(\phi)$. If
\begin{equation}\label{E:suto}
\forall x\in {\mathcal P}(\omega)\ \big(\sup_{n\in x} \phi(\{ n\}) <\infty \Rightarrow \phi(x)<\infty\big),
\end{equation}
then $\Id$ is generated by the sets
\[
\{ n\in \omega \colon \varphi(\{ n\}) <M\} \in \Id
\]
with $M\in\omega$. Thus, $\Id$ is countably generated.

Assume, therefore, that \eqref{E:suto} fails, which allows us to fix $x_0\in {\mathcal P}(\omega)$ such that
\begin{equation}\label{E:suin}
\sup_{n\in x_0}\varphi(\{ n\}) <\infty \;\hbox{ and }\;   \varphi(x_0)=\infty.
\end{equation}
Set
\begin{equation}\label{E:mze}
M_0=\sup_{n\in x_0}\varphi(\{ n\}).
\end{equation}
Using \eqref{E:suin}, the definition \eqref{E:mze}, and the semicontinuity of $\varphi$,
we can find,
for each $M>2M_0$,  a sequence $(a^M_i)_{i\in\omega}$ of finite subsets of $x_0$ such that, for all $i$,
\begin{equation}\label{E:ama}
\max a^M_i < \min a^M_{i+1}\;\hbox{ and }\; M-M_0\leq \varphi(a^M_i)<M.
\end{equation}
By flatness of $\varphi$, we can find $N_M>0$ and a subsequence $(a^M_{i_j})$ of $(a^M_i)$ such that, for each $t\in \omega$,
$\varphi\big(\bigcup_{j\leq t} a^M_{i_j}\big) < N_M$, from which, by semicontinuity of $\varphi$, we get
\begin{equation}\label{E:mmm}
\varphi\big(\bigcup_j a^M_{i_j}\big) \leq N_M.
\end{equation}
For ease of notation, we assume that the whole sequence $(a^M_i)$ has property \eqref{E:mmm}.

Towards a contradiction, suppose that there is a topology $\tau$ on $\Id$ that makes $\Id$ into a basic partial order.
For $k\in \omega$, let
\[
x^M_k= \bigcup_{i\geq k} a^M_{i}
\]
Note that, by \eqref{E:mmm}, we have that $x^M_k\in \Id$; moreover, the sequence $(x^M_k)$ is bounded in $\Id$ by $x^M_0$. Thus,
it has a subsequence $(x^M_{k_n})_n$ that is convergent with respect to $\tau$. By the fact that $\tau$ is basic, for each $k$, the set ${\mathcal P}(x_k^M)$ is
a $\tau$-compact subset of $\Id$. It follows that the limit of $(x^M_{k_n})_n$ is an element of ${\mathcal P}(x_k^M)$ for each $k$; thus,
\[
x_{k_n}^M \to \emptyset\hbox{ with respect to }\tau, \hbox{ as }n\to\infty
\]
since $\bigcap_k {\mathcal P}(x_k^M) =\{\emptyset\}$.
Since $\tau$ is a metric topology, we can find a diagonal sequence $(y_M)_{M\in\omega}$ convergent to $\emptyset$, that is, for each $M$, there exists
$n$ with $y_M= x_{k_n}^M$ and $y_M\to\emptyset$ with respect to $\tau$ as $M\to\infty$. Again, by using the fact that $\Id$ is basic with $\tau$, the convergent sequence
$(y_M)_M$ has a bounded subsequence. But this is impossible, since by the second part of \eqref{E:ama} and the definitions of $x^M_k$ and $y_M$,
we have $\phi(y_M)\geq M-M_0$ for each $M$.
\end{proof}

\section{A dichotomy for flat ideals}\label{FGF}

We prove a dichotomy theorem that is the starting point of our investigation of flat ideals. A result of this form was proved
in \cite{HRZ} for fragmented ideals. Here we extend it to flat ideals building on the proof from \cite{HRZ}.
The theorem asserts that flat ideals that are not gradually flat are of the highest possible cofinal type. Note that the ideal of subsets of $\omega$ from
\cite[Section 5, Proposition 3]{LV} that is Tukey equivalent to $[\mathfrak c]^{<\omega}$ is easily seen to be flat.
In the next section, in Theorem \ref{thm8}, we illuminate the second half of this dichotomy by giving several conditions on flat ideals equivalent to gradual flatness.

A lsc submeasure $\varphi$ is said to be {\bf gradual} if, for each $M>0$, there exists $N>0$ such that, for each $l$, we have
\[
\exists j \ \forall B\in [\mathcal P(\omega\setminus j)]^{\leq l}\, \Big(
  \max_{b\in B}\varphi(b)< M\Rightarrow \varphi(\bigcup B)< N\Big).
\]
We call a lsc submeasure  $\varphi$  {\bf gradually flat} if it is both flat and gradual, and we say that an ideal $\mathcal I$ is {\bf gradually flat} if $\mathcal I={\rm Fin}(\varphi)$ for some  gradually flat lsc submeasure $\varphi$.

\begin{lemma}
If $\varphi,\psi $ are flat  and $\Id={\rm Fin}(\varphi)={\rm Fin}(\psi)$, then  $\varphi$ is
gradually flat if and only if $\psi$ is gradually flat.
\end{lemma}

\begin{proof}
First, we prove that, for each $K_1>0$, there exists $K_2>0$ such that, for each $b\in {\mathcal P}(\omega)$,
\begin{equation}\label{E:bb}
\varphi(b)\leq K_1\Rightarrow \psi(b)\leq K_2\;\hbox{ and }\;\psi(b)\leq K_1\Rightarrow \varphi(b)\leq K_2.
\end{equation}
Otherwise, there exists $K>0$ such that, for each $n\in\omega$, there exists a set $b_n\in{\mathcal P}(\omega)$ with 
either  $\varphi(b_n)\leq K $  and  $\psi(b_n)> n$, or
 $\psi(b_n)\leq K  $  and  $\varphi(b_n)> n$.  
By pigeonhole principle and  compactness of ${\mathcal P}(\omega)$, 
we can assume that 
\begin{equation}\label{E:in1}
\psi(b_n)\leq K,\; \varphi(b_n)> n,\hbox{ for all }n, \hbox{ and }(b_n)_{n} \hbox{ converges to }b\in {\mathcal P}(\omega).
\end{equation}
Since $\psi(b_n)\leq K$ for all $n$, by  lower semicontinuous of $\psi$, we have $\psi(b)\leq K$, 
from which, by ${\mathcal I} ={\rm Fin}(\varphi)={\rm Fin}(\psi)$,  we get 
\begin{equation}\label{E:in2}
\varphi(b)<\infty.
\end{equation}
Let  $N$ witness flatness of $\psi$ for $K+1$.
Using \eqref{E:in1}, \eqref{E:in2}, and lower semicontinuity of $\varphi$, one recursively constructs finite sets 
$c_n\in \mathcal P(b_{j_n}\setminus b)$, $n\in \omega$, for some increasing sequence $(j_n)_n$, so that
\[
\psi(\bigcup_{i\leq n}c_i)< N\hbox{ and }\varphi(c_n)> n,\hbox{ for all }n\in \omega.
\]
The above inequalities give $\psi(\bigcup_{n\in \omega}  c_n )\leq N$ and $\varphi(\bigcup_{n\in \omega }c_n)=\infty$, contradicting 
${\rm Fin}(\varphi)={\rm Fin}(\psi)$, and therefore proving \eqref{E:bb}.

Assume $\psi$ is not gradual. This assumption allows us to fix
$M>0$, for which the following condition holds. For each $N>0$, we can find $l_N$ and a sequence 
$(B^N_n)_n$ of finite families of subsets of $\omega$ so that: 
$|B^N_n|\leq l_N$, $\min \{ \min b\colon b\in B^N_n\} \to \infty$ as $n\to \infty$,
$\max_{b\in B^N_n}\psi(b)< M$, and $\psi(\bigcup B^N_n)\geq N$.
It follows now from the second part of \eqref{E:bb} that there exists $M'>0$ such that
$\max_{b\in B^N_n}\varphi(b)< M'$ for all $n, N$. From the first part of \eqref{E:bb}, it follows that there exists
a function $g$ with $g(N)\to \infty$ as $N\to \infty$ such that $\varphi(\bigcup B^N_n)> g(N)$, for all $n, N$. 
Thus, $\varphi$ is not gradual, and the lemma is proved. 
\end{proof}

The following result is the promised dichotomy theorem.  Recall that if 
a directed partial order contains a strongly unbounded set of size $\mathfrak c$, then it 
is Tukey equivalent to  $[\mathfrak c]^{<\omega}$. 
Here, a subset $X$ of a partial order is {\bf strongly unbounded} if every infinite subset of $X$ is unbounded.

\begin{theorem}\label{dichotomy} Let $\Id$ be a flat ideal. Then either
$\Id \equiv_T [\mathfrak c]^{<\omega}$ or  $\Id$ is gradually flat.
\end{theorem}

\begin{proof} Let $\Id = {\rm Fin}(\varphi)$ for a flat submeasure $\varphi$, and assume that $\varphi$ is not gradual. This assumption allows us to fix
$M>0$ such that for each $N$ there exists $\ell_N$ and a sequence $(B^N_n)$ of families of subsets of $\omega$ such that, for all $n, N\in\omega$,
\begin{equation}\label{E:BB}
|B^N_n| \leq \ell_N,\; n\leq \min_{b\in B^N_n} b,\; \max_{b\in B^N_n}\varphi(b)<M,\hbox{ and } \varphi(\bigcup B^N_n)\geq N.
\end{equation}
By semicontinuity of $\varphi$, we can assume that the sets in $B^N_n$ are finite for each $n, N$. By going to subsequences of 
$(B^N_n)$, we can assume that, for each $n$,
\[
|B^N_n| = \ell_N
\]
and
\[
(n,N)\not= (n', N')\Rightarrow \left(\max \bigcup B^N_n<\min \bigcup B^{N'}_{n'}\,\hbox{ or }\,\max \bigcup B^{N'}_{n'}<\min \bigcup B^N_n\right).
\]
Let $K_M$ witness flatness of $\varphi$ for $M$. By the choice of $K_M$ and by \eqref{E:BB},
we see that, for each $p\in \omega$, the following statement holds:
for each $N$, for large enough $n$, if
$a\subseteq p$ is such that $\varphi(a)<K_M$ and $b\in B^N_n$, then $ \varphi(a\cup b) <K_M$.
Using this observation and the inequality $\varphi(\emptyset)=0<K_M$,
a recursive construction allows us to pick, for each $N$,  an infinite subset $D^N$ of $\omega$ with the following property: for each finite sequence
$(b_i)_{i<t}$, with $t\in\omega$, of sets picked from distinct $B^N_n$ with $N\in \omega$ and $n\in D^N$, we have
\[
\varphi( \bigcup_{i<t} b_i)<K_M.
\]
Thus, by going to subsequences,
we can assume that the above inequality holds for all finite sequences $(b_i)_{i<t}$ of sets picked from distinct $B^N_n$ with $n, N\in \omega$.
Now, lower semicontinuity of $\varphi$ gives that, for each infinite sequence $(b_i)_i$ of sets picked from distinct $B^N_n$ with $n, N\in \omega$, we have
\begin{equation}\label{E:unb}
\varphi( \bigcup_i b_i)\leq K_M.
\end{equation}

From this point on, we follow the lines of the proof of \cite[Theorem~2.4]{HRZ} with appropriate modifications.
Let $C^k$, $k\in\omega$, be a partition of $\omega$ into infinite sets.
For each $N$, let $(f_p^N)_p$ be a sequence of functions
\[
f_p^{k,N}\colon C^k\to \bigcup_{n\in C^k} B^N_n
\]
such that $f_p^{k,N}(n)\in B^N_n$ for all $n\in C^k$ and, for all $k$ and $a\in [\omega]^{l_N}$,
\begin{equation}\label{E:bn}
\{ f^{k,N}_p(n)\colon p\in a\} = B^N_n,\hbox{ for some }n\in C^k.
\end{equation}
Such sequences can be found by \cite[Claim 2.5, p.\,33]{HRZ}.
Now for each $p$, we let
\[
J^k_p= \bigcup_N \bigcup_{n\in C^k}  f^{k,N}_p(n).
\]
From the definition of $J^k_p$, by \eqref{E:bn} and \eqref{E:BB}, we get that, for each $N$, 
\begin{equation}\label{E:sms}
\varphi(\bigcup_{p\in a} J^k_p)\geq N, \hbox{ for all }k\hbox{ and }a\in [\omega]^{l_N}.
\end{equation}
On the other hand, observe that, for each $g\in \omega^\omega$, the set $\bigcup_k J^k_{g(k)}$ is the union of a sequence $(b_i)$, where sets $b_i$ are chosen from distinct
$B^N_n$ with $n, N\in \omega$. Thus,
\eqref{E:unb} implies that
\begin{equation}\label{E:gg}
\varphi(\bigcup_k J^k_{g(k)} ) \leq N_M, \hbox{ for each }g\in \omega^\omega.
\end{equation}

Let ${\mathcal A}\subseteq \omega^\omega$ be a perfect family of eventually different elements of $\omega^\omega$.
For $g\in {\mathcal A}$, define
\[
G(g)= \bigcup_k J^k_{g(k)}.
\]
It is clear that $G\colon {\mathcal A}\to 2^\omega$ is a continuous injection and
\eqref{E:gg} implies that $G(g)\in \Id$ for each $g\in {\mathcal A}$.

We check that the family $\{ G(g)\colon g\in {\mathcal A}\}$ is strongly unbounded in $\Id$.
Let $g_i$, $i<l_N$, be distinct elements of $\mathcal A$. Then,
there exists $k\in \omega$ such that $g_i(k)$ are all distinct for $i<l_N$. Therefore, by \eqref{E:sms}, we get
\[
\varphi\left(\bigcup_{i<l_N} G(g_i)\right)\geq \varphi(\bigcup_{i<l_N} J^k_{g_i(k)})\geq N.
\]
Thus, if $g_i\in {\mathcal A}$, $i\in\omega$, are all distinct, then
\[
\bigcup_i G(g_i)\not\in \Id,
\]
and the conclusion follows.
\end{proof}

\section{Characterizations of gradually flat ideals}

In this section we give various characterizations of gradually flat ideals. To formulate our theorem, we need to introduce some notions.

Recall that a subset $X$ of a partial order $P$ is called {\bf weakly bounded} if 
every infinite subset of $X$ contains an infinite bounded subset. (Note that this condition is equivalent to $X$ not containing an infinite strongly unbounded subset, where strong unboundedness is defined in Section~\ref{FGF}.) 
Following \cite{LV}, we say that a partial order $P$ is {\bf $\sigma$-weakly bounded}
if $P$ is the union of a countable family of weakly bounded sets.

Let $\Id$ be an ideal of subsets of $\omega$. We will use another game introduced by Fremlin in \cite{fremlin}. 
Given an ideal $\Id$, the game $\Gamma_1(\Id)$ is a two player game, where the $n$-th move for each player is described  as follows.

\smallskip
\noindent By convention, $U_{-1} =\Id$. Given  $U_{n-1}\subseteq \Id$,

Player I plays a countable cover $\mathcal{U}_n$ of $U_{n-1}$,

Player II responds with $U_n\in\mathcal{U}_n$ and a sequence $\ang{x^n_i}_i$ in $U_n$,

Player I then chooses a subsequence by playing $A_n\in[\omega]^{\omega}$,

Player II plays the last bit by choosing $m_n\in\omega$.

\smallskip

\noindent More graphically, the $n$-th move is represented as follows.

\bigskip

\begin{tabular}{cccccccc}
\multicolumn{1}{c|}{$I$} &$\mathcal U_{n}$ countable cover of $U_{n-1}$ &&$A_n\in [\omega]^\omega$    \\ \hline
\multicolumn{1}{c|}{$II$}& &$U_n\in \mathcal U_n$,  $\ang{x^n_i}_i\subseteq U_n$& &$m_n\in\omega$  \\
\end{tabular}\\

\bigskip
After the $n$-th move has been made, we define
\begin{equation}\label{E:baga}
b_n=\bigcup\{x^n_i:i\in A_n\cap m_n\}\in\mathcal I.
\end{equation}
We declare that Player I wins if the sequence $\ang{b_n}_n$ is bounded in $\Id$, that is, if $\bigcup_n b_n\in \Id$;
otherwise Player II wins.

We say that an ideal $\Id$ is {\bf Fremlin} if Player I has a winning strategy in the game $\Gamma_1(\Id)$.

Two  results from \cite{fremlin} concerning Fremlin ideals relevant to us are:
\begin{enumerate}
\item[---] being Fremlin is a property that is downward closed under Tukey reduction;
\item[---] the density zero ideal ${\mathcal Z}_0$ is Fremlin; hence, if $\Id$ is Tukey below $\mathcal{Z}_0$, then $\Id$ is Fremlin.
\end{enumerate}

We can now state our characterization theorem.

  \begin{theorem}\label{thm8} If $\mathcal I$ is a flat ideal, then the following conditions are equivalent.
  \begin{enumerate}
  \item[(i)] $\mathcal I$ is gradually flat.
   \item[(ii)] $\mathcal I\leq_T \mathcal Z_0$.
   \item[(iii)] $\mathcal I\leq_T l_1$.
    \item[(iv)] $\mathcal  I \not \equiv_T [\mathfrak c]^{<\omega}$.
\item[(v)] $\mathcal I$ is Fremlin.
   \item[(vi)]  $\mathcal I$ is $\sigma$-weakly bounded.
   \item[(vii)] $\omega^\omega\not\leq_T \mathcal I$.
  
\end{enumerate}
 \end{theorem}

In relation to considering gradually flat ideals with the Tukey order,
we point out that it follows from \cite[Proposition~5.28]{Ma} that the class of gradually flat ideals equipped with $\leq_T$ is rich; namely, the
quasi-order $({\mathcal P}(\omega), \subseteq^*)$, where $\subseteq^*$ is inclusion modulo finite sets, embeds into it.
To see this, it suffices to notice that the ideals considered in \cite[Proposition~5.28]{Ma} are flat by their definitions, and
they are gradually flat by \cite[formula (5.4)]{Ma}.

Theorem~\ref{thm8} above raises some natural questions.

\begin{question} Let $\mathcal I$ be an $F_\sigma$ ideal of subsets of $\omega$.
\begin{enumerate}
    \item[(1)] Is $\Id\equiv_T [\mathfrak c]^{<\omega}$ or $\Id\leq_T l_1$?
    \item[(2)] Assume $\omega^\omega\leq_T \Id$. Is $l_1\leq_T\Id$?
    \item[(3)] (\cite{LV}) Assume $\omega^\omega\not\leq_T \Id$. Is $\Id$ $\sigma$-weakly bounded?
\end{enumerate}
\end{question}
It may be worth mentioning in the context of the question above that, by \cite{LV}, if $\Id$ is an analytic ideal of subsets of $\omega$ and
$\omega^\omega\not\leq_T \Id$, then $\Id$ is $F_{\sigma}$.

In our proof of Theorem~\ref{thm8}, we start with a lemma connecting Fremliness and $\sigma$-weak boundedness. The lemma is somewhat more general than needed
in order to prove Theorem~\ref{thm8}, and may be of some independent interest.

\begin{lemma}\label{fimpliesswb}
If $\Id$ is a Fremlin $F_\sigma$ ideal, then $\Id$ is $\sigma$-weakly bounded.
\end{lemma}

\begin{proof}
Let $\varphi$ be a lower semicontinuous submeasure such that $\Id=\rm{Fin}(\varphi)$, and let $\tau$ be a winning strategy for Player I in $\Gamma_1(\Id)$.
Assume, in order to reach a contradiction, that $\Id$ is not $\sigma$-weakly bounded. We will play against $\tau$ and produce a play in which Player II wins:
$\Id=U_0$ is by assumption not $\sigma$-weakly bounded. After move $n-1$ has been made, assume that $U_{n-1}$ is chosen to be  
the restriction of $I$ to $U_{n-1}$ is not $\sigma$-weakly bounded. Now, in the $n$-th move, if $\cU_n=\tau(U_{n-1})$ is a covering for $U_{n-1}$, there must be $U_n\in\cU_n$ that is not $\sigma$-weakly bounded,
we play such $U_n$ together with a strongly unbounded sequence $(x^n_i)_i$ in $U_n$. Once $\tau$ has chosen $A_n\in[\omega]^{\omega}$, as any  subsequence of
$(x^n_i)_i$ is strongly unbounded, we can find $m_n$ large enough so that $\varphi(b_n)>n$, where $b_n$ is as in \eqref{E:baga}.
At the end, this play against $\tau$ would have produced the sequence
$(b_n)_n$ which is unbounded in $\Id$.
\end{proof}

Now, we prove a lemma on gradually flat ideals that will be our main technical tool when dealing with such ideals in the proof of Theorem~\ref{thm8}.

\begin{lemma}\label{gradual}
Let $\Id$ be gradually flat.
There are sequences $\mathcal{X}_k, {\mathcal Y}_k$, $k\in\omega$, of downward closed subsets of $\Id$ such that
\begin{enumerate}
\item[(i)] $\bigcup_k {\mathcal X}_k =\Id$;

\item[(ii)] $\overline{{\mathcal Y}_k} \subseteq \Id$, for each $k$;

\item[(iii)] for each $k, l$, there exists $g_{k,l}\in\omega^\omega$ with
\[
\forall n \ \big(\forall a\in \mathcal{Y}_{k}\cap \mathcal{P}(n)\big)\ \big(\forall B\in
[\mathcal{X}_k\cap {\mathcal P}\big(\omega\setminus g_{k,l}(n)\big)]^{\leq l}\big)\left(a\cup\bigcup B\in\mathcal{Y}_{k}\right).
\]
\end{enumerate}
\end{lemma}

\begin{proof} Fix a gradually flat submeasure $\varphi$ such that $\Id={\rm Fin}(\varphi)$. Given $k\in\omega$, there exists
$M_k>0$ as in the definition of gradual submeasure 
(with $k+1$ playing the role of $M$ and $M_k$ playing the role of $N$), for which
there is a function $h_k\in \omega^\omega$ such that
\[
\forall l\ \forall B\in [{\mathcal P}(\omega\setminus h_k(l))]^{\leq l} \big( \max_{b\in B}\varphi(b)< k+1\Rightarrow \varphi(\bigcup B)<M_k\big).
\]
Given this $M_k$, we find $N_k>0$ as in the definition of a flat submeasure, for which there is a function
$h_k'\in\omega^\omega$ such that for each $n$ and each $a\in {\mathcal P}(n)$
with $\varphi(a)<N_k$, we have
\[
\forall b\in {\mathcal P}(\omega\setminus h_k'(n)) \big(\varphi(b)<M_k\Rightarrow \varphi(a\cup b)<N_k\big).
\]

For $k\in\omega$, set
\[
{\mathcal X}_k = \{ x\in {\mathcal P}(\omega) \colon \varphi(x)<k+1\}\;\hbox{ and }\; {\mathcal Y}_k = \{ x\in {\mathcal P}(\omega) \colon \varphi(x)<N_k\}.
\]
Point (i) is obvious from the definition of ${\mathcal X}_k$. Point (ii) follows from the set 
$\{ x\in {\mathcal P}(\omega)\colon \varphi(x)\leq N_k\}$ being a closed subset of ${\mathcal P}(\omega)$, by semicontinuity of $\varphi$, and from
\[
{\mathcal Y}_k \subseteq \{ x\in {\mathcal P}(\omega)\colon \varphi(x)\leq N_k\}\subseteq \Id.
\]

Let
\[
g_{k,l}(n) =\max( h_k(l), h_k'(n)).
\]
Observe that $g_{k,l}$ has the following property
\begin{equation}\label{E:gl}
\begin{split}
\forall n &\forall a\in {\mathcal P}(n) \Big( \phi(a)<N_k \Rightarrow\\
&\forall B\in [{\mathcal P}(\omega\setminus g_{k,l}(n))]^{\leq l} \big( \max_{b\in B}\phi(b)< k+1\Rightarrow \varphi(a\cup \bigcup B)<N_k\big)\Big).
\end{split}
\end{equation}
Property \eqref{E:gl} implies (iii). The lemma follows.
\end{proof}

\begin{proof}[Proof of Theorem~\ref{thm8}]
$(i)\Rightarrow(ii)\;$ We will use the following ideal introduced in \cite{LV}. Let $\alpha\in\omega^\omega$, and define
\[
\mathcal{J}^{\alpha}=\{A\subseteq\omega\times\omega :\forall n\ A(n)\subseteq\alpha(n) \hbox{ and }|A(n)|/2^n\rightarrow 0\},
\]
where $A(n) = \{ m\in\omega\colon (n,m)\in A\}$. (Such an ideal is called ${\mathcal J}^{\alpha, (2^n)}( c_0)$ in \cite{LV}. For convenience, we shortened this piece of notation.)
It was proved in \cite[Proposition~4(b) and Remark pp.\,186,\,187]{LV} that
\begin{equation}\label{E:tuz}
\mathcal{J}^{\alpha}\leq_T {\mathcal Z}_0.
\end{equation}

Now, given a sequence $\alpha_k\in \omega^\omega$, $k\in\omega$, we consider the ideal
\[
\bigoplus_k {\mathcal J}^{\alpha_k} = \{ A\subseteq \omega\times\omega\times\omega \colon \big( \forall k \ A(k) \in {\mathcal J}^{\alpha_k}\big) \hbox{ and } \big(\exists k_0\forall k\geq k_0 \ A(k)=\emptyset\big) \},
\]
where $A(k) = \{ (n,m)\in\omega\colon (k, n,m)\in A\}$.
We argue that 
\begin{equation}\label{E:sumi}
\bigoplus_k {\mathcal J}^{\alpha_k}\leq_T{\mathcal Z}_0.
\end{equation}
Indeed, for each $k$, fix a Tukey map $f_k\colon {\mathcal J}^{\alpha_k} \to {\mathcal Z}_0$ that exists by \eqref{E:tuz}. For $A\in \bigoplus_k {\mathcal J}^{\alpha_k}$,
let $k_A$ be the smallest element of $\omega$ such that $A(k)=\emptyset$ for all $k\geq k_A$. Define
\[
F(A)= k_A\cup \bigcup_{k< k_A} f_k\big(A(k)\big).
\]
Obviously, $F\colon \bigoplus_k {\mathcal J}^{\alpha_k}\to {\mathcal Z}_0$, and it is easy to check that it is a Tukey map.

By \eqref{E:sumi}, it suffices to show that, for appropriately chosen $\alpha_k$, $k\in \omega$, we have
\begin{equation}\label{E:JZ}
\Id\leq_T \bigoplus_k {\mathcal J}^{\alpha_k}.
\end{equation}

We define $\alpha_k\in\omega^\omega$, for $k\in\omega$.
Let $g_{k,l}$, ${\mathcal X}_k$, and ${\mathcal Y}_k$, $k,l\in\omega$, be as in the conclusion of Lemma~\ref{gradual}. Set
\[
m^k_0=0 \;\hbox{ and }\; m^k_{l} = g_{k, 2^{l+1}}(m^k_{l-1}), \hbox{ for }l>0.
\]
Then, by Lemma~\ref{gradual}(iii), the sequence $(m^k_l)$ has the following property for each $l>0$
\begin{equation}\label{E:long0}
\begin{split}
\forall a\in {\mathcal P}(m^k_{l-1}) \ \forall B\in [{\mathcal P}(\omega\setminus m^k_{l})]^{\leq 2^{l+1}} \big( ( a\in {\mathcal Y}_k\; &\hbox{and }B\subseteq {\mathcal X}_k)\\
&\Rightarrow a\cup \bigcup B \in {\mathcal Y}_k\big).
\end{split}
\end{equation}
Let $B_l$, for $l\in\omega$, and $l_0$ be such that
\begin{equation}\label{E:BL0}
B_l\subseteq {\mathcal X}_k\cap {\mathcal P}\big([m^k_l, m^k_{l+1})\big)\hbox{ and } |B_l|\leq 2^{l+1}, \hbox{ for all }l\geq l_0.
\end{equation}
By using \eqref{E:long0} recursively and  then taking the union, we get that 
\begin{equation}\label{similar} 
\bigcup_{l>l_0/2} \bigcup B_{2l} \in \overline{{\mathcal Y}_k}\;\hbox{ and }\; \bigcup_{l\geq l_0/2} \bigcup B_{2l+1} \in \overline{{\mathcal Y}_k}.
\end{equation}
Since the closure of $\mathcal Y_k$ is contained in $\mathcal I$, we have $\bigcup_{l} \bigcup B_{l} \in \Id$, from which we draw the following immediate conclusion
\begin{equation}\label{E:bds0}
\{ x\in {\mathcal P}(\omega)\colon \forall l \ \big( x\cap [m^k_l, m^k_{l+1}) \in B_l\big)\} \hbox{ is bounded in }\Id.
\end{equation}
For $l\in\omega$, let
\[
I^k_l= [m^k_l, m^k_{l+1}).
\]
Fix a one-to-one function $\pi_k:\bigcup_l\mathcal{P}(I^k_l)\to \omega$, and define
$\alpha_k\in \omega^\omega$ by letting
\[
\alpha_k(l)=1+\max \pi_k\big(\mathcal{P}(I^k_l)\big).
\]

Now we produce a function $\Psi\colon\Id \to \bigoplus_k {\mathcal J}^{\alpha_k}$.
For $x\in\mathcal{I}$, we have $\varphi(x)<\infty$, which allows us to define 
$k_x\in\omega$ to be the smallest $k$ with $x\in {\mathcal X}_k$.
Let $A_x\subseteq \omega\times\omega$ be defined by
\[
(l,m)\in A_x\Leftrightarrow x\cap I^{k_x}_l\not=\emptyset\hbox{ and }m=\pi(x\cap I^{k_x}_l).
\]
Observe that for every $x$, $A_x\in \mathcal{J}^{\alpha_{k_x}}$. Define $\Psi$ by
$\Psi(x)= \{ k_x\}\times A_x$. It is clear that $\Psi\colon \Id\to \bigoplus_k {\mathcal J}^{\alpha_k}$.

We claim that $\Psi$ is Tukey. Let $A\in \bigoplus_k\mathcal{J}^{\alpha_k}$ and
\[
\mathcal{A}=\{x\in\mathcal{I}:\Psi(x)\subseteq A\}.
\]
We need to see that $\bigcup {\mathcal A}\in \Id$.
As $A\in \bigoplus_k\mathcal{J}^{\alpha_k}$, we have that $A(k)=\emptyset$ for all but finitely many $k$. Thus, it suffices to see that for a fixed $k$
\begin{equation}\notag
\bigcup \{ x\in {\mathcal X}_k \colon A_x\subseteq A(k)\}\in \Id.
\end{equation}
Fix $k$, and set $B=A(k)$, $I_l = I^k_l$, and
${\mathcal B}= \{ x\in {\mathcal X}_k\colon A_x\subseteq B\}$.
We need to see that
\begin{equation}\label{E:sata}
{\mathcal B}\hbox{ is bounded in }\Id.
\end{equation}
Since $B\in {\mathcal J}^{\alpha_k}$,
there is $l_0$ such that for all $l\geq l_0$, $|B(l)|<2^l$. So, for each $l\geq l_0$,
\[
|\{x\cap I_l: x\in {\mathcal B}\}|=|\{m: \exists x\in {\mathcal B}\ \pi(x\cap I_l)=m \}|\leq|B(l)|<2^l.
\]
Hence, for each $l\geq l_0$, if $|\{x\cap I_l: x\in {\mathcal B}\}|<2^l$, and,
therefore, by \eqref{E:bds0}, we get \eqref{E:sata}, as required.

$(ii)\Rightarrow(iii)\;$ follows from $\mathcal Z_0\leq_T l_1$, see \cite{LV}.

$(iii)\Rightarrow(iv)\;$ follows from $l_1<_T [\mathfrak c]^{<\omega}$, see \cite{LV}.

$(iv)\Rightarrow(i)\; $  Since $\Id$ is flat, by Theorem~\ref{dichotomy}, $\Id$ is gradually flat.

$(i)\Rightarrow(v)\;$ We describe a winning strategy for Player I  in $\Gamma_1(\mathcal I)$. 
Let the sets $\mathcal{X}_k$ and ${\mathcal Y}_k$, $k\in\omega$, be as in the conclusion of Lemma~\ref{gradual}. By Lemma~\ref{gradual}(i), the family
$\{ {\mathcal X}_k\colon k\in \omega\}$ forms a covering of $\Id$.
In move $0$, Player I plays this covering. Player II responds by picking a set ${\mathcal X}_k$ and a sequence $(x_i)$ of elements of ${\mathcal X}_k$. So $k$ is fixed by Player II.
For the fixed $k$, we perform the following analysis.
Let $g_{k,l}$, $k,l\in\omega$, be as in the conclusion of Lemma~\ref{gradual}. Since $k$ is fixed, we set $g_l=g_{k,l}$. Define
\[
m_0=0 \;\hbox{ and }\; m_{l} = g_{(l+1)^2+1}(m_{l-1}), \hbox{ for }l>0.
\]
Then, by Lemma~\ref{gradual}(iii), the sequence $(m_l)$ has the following property for each $l>0$
\begin{equation}\label{E:long}
\begin{split}
\forall a\in {\mathcal P}(m_{l-1}) \ \forall B\in [{\mathcal P}(\omega\setminus m_{l})]^{\leq (l+1)^2+1} \big( ( a\in {\mathcal Y}_k\; &\hbox{and }B\subseteq {\mathcal X}_k)\\
&\Rightarrow a\cup \bigcup B \in {\mathcal Y}_k\big).
\end{split}
\end{equation}
Let $B_l$, for $l\in\omega$, be such that
\begin{equation}\label{E:BL}
B_l\subseteq {\mathcal X}_k\cap {\mathcal P}\left([m_l, m_{l+1})\right)\hbox{ and } |B_l|\leq (l+1)^2+1.
\end{equation}
By an argument similar to the one justifying \eqref{similar}, it follows from \eqref{E:long} that
\[
\bigcup_{l>0} \bigcup B_{2l} \in \overline{{\mathcal Y}_k}\;\hbox{ and }\; \bigcup_{l} \bigcup B_{2l+1} \in \overline{{\mathcal Y}_k}.
\]
Thus, by Lemma~\ref{gradual}(ii), we have $\bigcup_{l} \bigcup B_{l} \in \Id$, from which we get
\begin{equation}\label{E:bds}
\{ x\in {\mathcal P}(\omega)\colon \forall l \ \big( x\cap [m_l, m_{l+1}) \in B_l\big)\} \hbox{ is bounded in }\Id.
\end{equation}
Now Player I plays a convergent subsequence $(y^0_i)_{i<\omega}$ of $(x_i)_{i<\omega}$ so that
\begin{equation}\label{E:suse}
\forall l\ |\{ y^0_i\colon i\in \omega\} \cap [m_l, m_{l+1})| \leq l+1.
\end{equation}

Assume players I and II are about to make move $n+1$. Assume we have a sequence of
sets $a_l\subseteq [m_l, m_{l+1})$, $l<n$, such that the set played by Player II in move $p$ with $1\leq p\leq n$ is
\[
U_p= \{ x\in {\mathcal X}_k\colon \forall l<p\ x\cap [m_l, m_{l+1}) = a_l\}.
\]
Now, in move $n+1$, Player I plays the family of sets
\[
V_a = \{ x\in U_n\colon x\cap [m_n, m_{n+1}) = a\}.
\]
Then, still in move $n+1$, Player II picks one of these sets, which amounts to picking
$a_{n+1} \subseteq [m_n, m_{n+1})$. We let $U_{n+1}= V_{a_{n+1}}$. Further, Player II plays an arbitrary sequence
$(x_i)$ in $U_{n+1}$. Then Player I picks a convergent subsequence $(y^{n+1}_i)_{i<\omega}$ of $(x_i)_{i<\omega}$ so that
\begin{equation}\label{E:suse2}
\forall l\geq n+1 \ |\{ y^{n+1}_i\colon i\in \omega\} \cap [m_l, m_{l+1})| \leq l+1.
\end{equation}
This concludes our description of a strategy for Player I. We claim this is a winning strategy.

Indeed, after a run of the game is finished, we set
\[
B_l = \{ y^{n}_i\colon n, i\in \omega\} \cap [m_l, m_{l+1}).
\]
Clearly, $B_l\subseteq {\mathcal X}_k$. It follows from \eqref{E:suse} and \eqref{E:suse2} that $|B_l|\leq l^2+1$. Thus, condition \eqref{E:BL} is fulfilled.
Therefore, the set
\[
\{ y^n_i\colon n,i\in\omega\}
\]
is bounded in $\Id$ as it is included in
\[
\{ x\in {\mathcal P}(\omega)\colon \forall l \ \big( x\cap [m_l, m_{l+1}) \in B_l\big)\}
\]
which is bounded in $\Id$ by \eqref{E:bds}. It follows that Player I wins.

$(v)\Rightarrow(vi)$ follows from Lemma~\ref{fimpliesswb}.

$(vi)\Rightarrow(vii)$ follows directly from \cite[Theorem~1]{LV}.

$(vii)\Rightarrow(i)$ follows from Theorem~\ref{dichotomy}.
\end{proof}

\end{document}